\documentclass{amsart}

\usepackage{amsmath} 
\usepackage{amssymb}

\newcount\skewfactor
\def\mathunderaccent#1#2 {\let\theaccent#1\skewfactor#2
\mathpalette\putaccentunder}
\def\putaccentunder#1#2{\oalign{$#1#2$\crcr\hidewidth
\vbox to.2ex{\hbox{$#1\skew\skewfactor\theaccent{}$}\vss}\hidewidth}}
\def\name{\mathunderaccent\tilde-3 }

\newcommand{\forces}{\Vdash}

\newcommand{\can}{{}^{\omega}2}

\newcommand{\rest}{{\restriction}}

\newcommand{\conc}{{}^\frown\!}
\newcommand{\vtl}{\vartriangleleft} 
\newcommand{\vare}{\varepsilon}

\newcommand{\lex}{{\rm lex}}

\newcommand{\cA}{{\mathcal A}}

\newcommand{\cB}{{\mathcal B}}

\newcommand{\bbC}{{\mathbb C}}

\newcommand{\bbP}{{\mathbb P}}

\newcommand{\mbR}{{\mathbb R}}

\newcommand{\bV}{{\mathbf V}}

\newtheorem{theorem}{Theorem}[section] 
\newtheorem{claim}{Claim}[theorem]
\newtheorem{lemma}[theorem]{Lemma}

\theoremstyle{definition}

\newtheorem{definition}[theorem]{Definition}

\theoremstyle{remark}

\begin{document}

\title{Not so many non-disjoint translations}

\author{Andrzej Ros{\l}anowski}
\author{Vyacheslav V. Rykov} 
\address{Department of Mathematics\\
University of Nebraska at Omaha\\
Omaha, NE 68182-0243, USA}
\email{aroslanowski@unomaha.edu}
\email{vrykov@unomaha.edu}

\date{November 10, 2017}

\begin{abstract}
We show that, consistently, there is a Borel set which has uncountably many  
pairwise very non-disjoint translations, but does not allow a perfect set of
such translations. 
\end{abstract}

\maketitle

\section{Introduction}
There is some interest in the literature in Borel sets admitting many
pairwise disjoint translations. For instance, Balcerzak, Ros{\l}anowski and
Shelah \cite{BRSh:512} studied the $\sigma$--ideal of subsets
of $\can$ generated by Borel sets with a perfect set of pairwise disjoint
translations.  In  this article we are interested in somewhat dual property
of Borel sets: many overlapping translations. 

If $B\subseteq \can$ is an uncountable Borel set, then it includes a perfect
set $P$, and then for $x,y\in P$ we have
\[{\bold 0}, x+y\in (P+x)\cap (P+y).\]
Consequently, every uncountable Borel subset of $\can$ has a perfect set of
pairwise non-disjoint translations. However, if we demand that the
intersections are more substantial, then the problem of many non-disjoint
translations becomes more interesting. One should notice that if
$x+b_0=y+b_1$ then also $x+b_1=x+b_0$, so if $x\neq y$ and $(B+x)\cap (B+y)$
is finite then $|(B+x)\cap (B+y)|$ must be even. 

Here we investigate the first non-trivial case when $(B+x)\cap (B+y)$ has
least 4 elements. We show that it is consistent with ZFC that there is a
$\Sigma^0_2$ subset $B$ of the Cantor space $\can$ such that  
\begin{itemize}
\item for some uncountable set $H\subseteq \can$, $|(B+h)\cap (B+h')|\geq
  4$ for all $h,h'\in H$, but 
\item for every perfect set $P\subseteq \can$ there are $x,x'\in P$ such that 

$|(B+x)\cap (B+x')|\leq 2$.
\end{itemize}
Our proof follows the spirit of the proof of Shelah \cite[Theorem
1.13]{Sh:522}, but since we cut on generality, our arguments are more
straightforward. We fully utilize the algebraic properties of $(\can,+)$, in
particular the fact that all elements of $\can$ are self-inverse. We do not
know if similar arguments can be made for $\mbR$ or even other product
topological groups (like ${}^\omega 4$ with coordinatewise addition modulo
4). 

This line of research will be continued in Ros{\l}anowski and Shelah
\cite{Sh:F1444}, where we will deal with the general case of $\kappa$ many
pairwise non-disjoint translations (getting the full parallel of
\cite[Theorem 1.13]{Sh:522}).  
\bigskip

\noindent {\bf Notation and Terminology}\qquad Our notation is rather
standard and compatible with that of classical textbooks (like Jech
\cite{J}). However, in forcing we keep the convention that {\em a stronger
  condition is the larger one}.  

Ordinal numbers will be denoted be the lower case initial letters of
the Greek alphabet $\alpha,\beta,\gamma,\delta,\vare$ and $\zeta,\xi$.
Natural numbers (finite ordinals) will be called $i,j,k$ and $\ell,n$. 

For a forcing notion $\bbP$, all $\bbP$--names for objects in
  the extension via $\bbP$ will be denoted with a tilde below (e.g.,
  $\name{h}$, $\name{T}$), and $\name{G}_\bbP$ will stand for the
  canonical $\bbP$--name for the generic filter in $\bbP$. 

For two sequences $\eta,\nu$ we write $\nu\vtl\eta$ whenever
  $\nu$ is a proper initial segment of $\eta$, and $\nu \trianglelefteq\eta$
  when either $\nu\vtl\eta$ or $\nu=\eta$.  A {\em tree\/} is a
  $\vtl$--downward closed set of sequences.  

The set of all sequences of length $n$ and with values in ${0,1}$ is
denoted by ${}^n 2$ and we let ${}^{\omega>} 2=\bigcup\limits_{n<\omega}
{}^n 2$. For $\sigma\in {}^{\omega>} 2$ let $[\sigma]=\{x\in\can: \sigma\vtl   
x\}$.  The Cantor space  $\can$ of all infinite 0--1 sequences is equiped
with the topology generated by sets of the form $[s]$ and the coordinatewise
addition $+$ modulo 2. Thus $(\can,+)$ is a topological group. 

For a tree $T\subseteq {}^{\omega>} 2$ we set $[T]=\{x\in\can: (\forall
n<\omega)(x\rest n\in T)\}$.  

For a set $A\subseteq X\times Y$ and $x\in X$ and $y\in Y$ let 
\[A_x=\{z\in Y: (x,z)\in A\}\quad \mbox{ and }\quad A^y=\{z\in X:
(z,y)\in A\}.\] 

\section{Some Technicalities}

\begin{definition}
\label{nice4}
  Let $1<\ell<\omega$. A {\em 4--arrangement\/} in  ${}^\ell 2$
  is a tuple $\langle a,b,c,d\rangle\subseteq {}^\ell 2$ such that
  $a<_\lex b<_\lex c<_\lex d$ and
\[\begin{array}{l}
\min\{k<\ell:a(k)\neq c(k)\}=\min\{k<\ell:b(k)\neq c(k)\}=\\
\min\{k<\ell:a(k)\neq d(k)\} =\min\{k<\ell:b(k)\neq d(k)\}.
\end{array}\]
\end{definition}

\begin{lemma}
\label{get4}
Let $15<\ell<\omega$. Suppose that $h:[{}^\ell 2]^2\longrightarrow 2$ is a
coloring with the property that
\begin{enumerate}
\item[$(\circledast)$] if $a,b,c\in {}^\ell 2$ are distinct, then $h(a,b)=1$
  or $h(a,c)=1$ or $h(b,c)=1$.
\end{enumerate}
(That is, there is no $h$-homogenous triangle in color $0$.) Then there is a
set $\cA\subseteq {}^\ell 2$ such that
\begin{enumerate}
\item$|\cA|\geq 5$,  and $\cA$ contains a 4--arrangement, and  
\item $\cA$ is $h$-homogeneous in color $1$, i.e., $h(a,b)=1$ for distinct
  $a,b\in \cA$.
\end{enumerate}
\end{lemma}

\begin{proof}
First, for $a\in {}^\ell 2$ let $Z_a=\{x\in {}^\ell 2: x\neq a\
\wedge\ h(x,a)=0\}$. It follows from the assumption $(\circledast)$
that 
\begin{enumerate}
\item[$(*)$] for each $a$, the set $Z_a$ is $h$-homogenous in color $1$. 
\end{enumerate}
If for some $a\in  {}^\ell 2$ the set $Z_a$ satisfies the requirements of
(1), then we are done. So suppose that 
\begin{enumerate}
\item[$(\odot)$] for each $a\in {}^\ell 2$ either $|Z_a|\leq 4$ or $Z_a$
  contains no $4$--arrangement.  
\end{enumerate}
Let $a\in  {}^\ell 2$ be the sequence constantly equal $0$ and let $d\in
{}^\ell 2$ be the $ <_\lex$--last element of $ {}^\ell 2\setminus Z_a$. It
follows from $(\odot)$ that $\{x\in  {}^\ell 2: x\rest (\ell-3)\equiv
1\}\setminus Z_a\neq \emptyset$, and hence $d(k)=1$ for $k<\ell-3$. 

Let $Y=\{\sigma\in  {}^{\ell-3} 2:\sigma(0)=0\ \wedge\ \sigma(1)=1\}$ and
for $\sigma\in Y$ let $X_\sigma=\{x\in  {}^\ell 2:\sigma\vtl x\}$. By
$(\odot)$, $X_\sigma\setminus Z_a\neq\emptyset$ (for each $\sigma\in Y$), so
we may pick $x_\sigma\in X_\sigma$ such that $h(a,x_\sigma)=1$. Again by
$(\odot)$, the set $\{x_\sigma:\sigma\in Y\}$ cannot be contained in $Z_d$,
so we may pick $\sigma^*\in Y$ such that $h(d,x_{\sigma^*})=1$. Set
$b=x_{\sigma^*}$ and note that $h(a,b)=h(b,d)=1$. 

Now we repeat the above procedure ``on $d$'s side'' for both $a$ and $b$ and
$d$. We let $Y'=\{\sigma\in {}^{\ell-3} 2: \sigma(0)=1\ \wedge\
\sigma(1)=0\}$ and $Y''=\{\rho\in {}^{\ell-6} 2: \sigma(0)=1\ \wedge\
\sigma(1)=0\}$ and $Y'''=\{\rho\in {}^{\ell-9} 2: \rho(0)=1\ \wedge\
\rho(1)=0\}$. For $\sigma\in Y'$ consider $X_\sigma=\{x\in {}^\ell
2:\sigma\vtl x\}$ and note that by our assumptions we may pick
$x_\sigma'\in X_\sigma$ such that $h(a,x_\sigma')=1$. Now, for each $\rho\in
Y''$ we may choose $\sigma_\rho\in Y'$ such that $\rho\vtl \sigma_\rho$
and $h(b,x_{\sigma_\rho}')=1$. By our assumptions, for some $\rho^*\in Y''$
we also have $h(d,x_{\sigma_{\rho^*}}')=1$. Set $c=x_{\sigma_{\rho^*}'}$ and
note that $\langle a,b, c,d\rangle$ is a 4--arrangement which is
homogenous in color $1$. 

Repeating the above procedure again, but starting with $Y^+=\{\sigma\in
{}^{\ell-3} 2: \sigma(0)=\sigma(1)=0\ \wedge\ \sigma(2)=1\}$,
going through levels $\ell-3$, $\ell-6$, $\ell-9$ and $\ell-12>3$ and
dealing with $a,b,c,d$ one may find $e\in {}^\ell 2$ such that
$\cA=\{a,b,c,d,e\}$ satisfies the demands (1)+(2). 
\end{proof}

\begin{lemma}
 \label{litlem}
Let $0<\ell<\omega$ and let $\cB\subseteq {}^\ell 2$ be a linearly
independent set of vectors (in $({}^\ell2,+,\cdot)$ over $(2,+_2,\cdot_2)$).  
\begin{enumerate}
\item[(a)] If $a,b,c\in {}^\ell 2$ are pairwise distinct and
  $\{a,b,c\}+\{a,b,c\} \subseteq \cB+\cB$, then for some pairwise distinct
  $\eta,\nu,\rho\in \cB$ we have 
\[a+b=\eta+\nu\quad\mbox{ and }\quad a+c=\eta+\rho.\] 
\item[(b)] If $\cA\subseteq {}^\ell 2$, $|\cA|\geq 5$ and $\cA+\cA\subseteq
  \cB+\cB$, then for a unique $x\in {}^\ell 2$ we have $\cA+x\subseteq \cB$. 
\end{enumerate}
\end{lemma}

\begin{proof}
(a)\quad Let $\nu_a,\nu_b,\eta_a,\eta_c,\rho_b,\rho_c\in\cB$ be such that  
\[a+b=\nu_a+\nu_b,\quad a+c=\eta_a+\eta_c,\quad \mbox{ and }\quad
b+c=\rho_a+\rho_c.\]    
Then $\nu_a\neq \nu_b$, $\eta_a\neq \eta_c$, $\rho_b\neq\rho_c$ and 
\[\rho_b+\rho_c=b+c=a+b+a+c=\nu_a+\nu_b+\eta_a+\eta_c.\] 
By the linear independence of $\cB$ we conclude $\{\nu_a,\nu_b\}\cap
\{\eta_a,\eta_c\} \neq\emptyset$. 
\medskip

\noindent (b)\quad Let $\cA=\{a_0,a_1,\ldots,a_{n-1}\}$, $n=|\cA|\geq
5$. Using clause (a) we may choose $\eta,\nu,\rho\in\cB$ such that 
\[a_0+a_1=\eta+\nu\quad \mbox{ and }\quad a_0+a_2=\eta+\rho.\]
Let $x=a_0+\eta$. We will argue that $a_i+x\in\cB$ for all $i<n$. Clearly by
our choices this holds for $i\leq 2$. Suppose $2<i<n$ is such that
$a_i+x\notin\cB$. Let $\eta^*,\eta^+\in\cB$ be such that $a_0+a_i=
\eta^*+\eta^+$. By clause (a) and our assumption on $i$ we have $\nu,\rho\in
\{\eta^*,\eta^+\}$, so $a_0+a_i=\nu+\rho$. Let $j<n$ be such that $j\notin
\{0,1,2,i\}$, and let $a_0+a_j=\nu^*+\nu^+$, $\nu^*,\nu^+\in\cB$. Then 
\[|\{\nu^*,\nu^+\}\cap \{\eta,\nu\}|=|\{\nu^*,\nu^+\}\cap \{\eta,\rho\}| =
|\{\nu^*,\nu^+\}\cap \{\nu,\rho\}|=1,\]
a contradiction. The uniqueness of $x$ follows from the linear independence
of $\cB$.   
\end{proof}

\begin{lemma}
\label{contread}
Suppose that $P\subseteq \can$ is a perfect set and $A_n\subseteq
P\times P$ are $\Sigma^1_1$ sets (for $n<\omega$) such that $P\times P
= \bigcup\limits_{n<\omega} A_n\cup \{(x,x):x\in P\}$. Then there is a
perfect set $P^*\subseteq P$ with the following property. 
\begin{enumerate}
\item[$(\heartsuit)$] For some increasing sequence of integers 
  $0=n_0<n_1<n_2<n_3<\ldots$, for each $k<\omega$ and any distinct
  $x,y\in P^*$ we have  
  \begin{enumerate}
  \item If $x\rest n_{k+1}\neq y\rest n_{k+1}$, then for all $x',y'\in
    P^*$
\[\big(x\rest n_{k+1}=x'\rest n_{k+1}\ \wedge\ y\rest n_{k+1}=y'\rest
n_{k+1}\big)\ \Rightarrow \big((x,y)\in A_k\ \Leftrightarrow\
(x',y')\in  A_k\big),\]  
\item the set $\{z\rest n_{k+1}: z\in P^*\ \wedge \ z\rest n_k=x\rest
  n_k\}$ has exactly two elements.  
  \end{enumerate}
\end{enumerate}
\end{lemma}

\begin{proof}
We will use the general result of Mycielski on the existence of
independent sets in topological algebras. To be able to quote his
theorem we have to introduce some definitions.

We say that a set $S\subseteq P^m$ is obtained by identification of
variables from $R\subseteq P^{m+1}$ if for some $i,j\leq m$ we have 
\[(x_1,\ldots,x_m)\in S\ \Leftrightarrow\ (x_1,\ldots,x_i,x_j,
x_{i+1},\ldots, x_m)\in R.\]

For $n<\omega$ let $J_n$ consist of all pairs
$(\sigma,\rho)\in {}^{\omega>} 2\times {}^{\omega>} 2$ such that the
set $([\sigma]\times[\rho])\cap A_n$ is meager (in $P^2$) and
$[\sigma]\cap [\rho]=\emptyset$, and let $K_n$ consist of pairs
$(\sigma,\rho)\in {}^{\omega>} 2\times {}^{\omega>} 2$ for which
$([\sigma]\times[\rho])\setminus A_n$ is meager and
$[\sigma]\cap [\rho]=\emptyset$. For $(\sigma,\rho)\in J_n$ we may fix
a Borel meager (in $P$) set $X_{\sigma,\rho}\subseteq P$ such that
\[\begin{array}{l}
(\forall x\in [\sigma]\setminus X_{\sigma,\rho})( (A_n)_x\cap 
[\rho]\mbox{ is meager})\mbox{ and}\\
(\forall y\in [\rho]\setminus X_{\sigma,\rho})( (A_n)^y\cap 
[\sigma]\mbox{ is meager}).
\end{array}\] 
Similarly, if $(\sigma,\rho)\in K_n$ then a Borel meager set
$X_{\sigma,\rho}\subseteq P$ is such that  
\[\begin{array}{l}
(\forall x\in [\sigma]\setminus X_{\sigma,\rho})( [\rho] \setminus
(A_n)_x \mbox{ is meager})\mbox{ and}\\
(\forall y\in [\rho]\setminus X_{\sigma,\rho})( [\sigma] \setminus
(A_n)^y \mbox{ is meager}).
\end{array}\] 
For $(\sigma,\rho)\in J_n\cup K_n$ let
\[\begin{array}{lr}
R^n_{\sigma,\rho}=\big\{(x_1,x_2,y_1,y_2)\in P^4:&
x_1,x_2\in [\sigma]\setminus X_{\sigma,\rho}\ \wedge \ y_1, y_2 \in
                    [\rho]\setminus X_{\sigma,\rho}\ \wedge \ \quad\\
            &(x_1,y_1)\in A_n\ \wedge \ (x_2,y_2)\notin A_n\ \big\}. 
\end{array}\]  
Clearly for every $(\sigma,\rho)\in J_n\cup K_n$ the set
$R^n_{\sigma,\rho}$ is meager (in $P^4$), moreover if $S\subseteq P^k$
is obtained from $R^n_{\sigma,\rho}$ by repeated identification and/or
permutation of variables, then $S$ is meager in $P^k$ as well. 

The sets $A_n$ have the Baire property and hence the sets $J_n\cup
K_n$ are dense in  ${}^{\omega>} 2\times {}^{\omega>} 2$.  Let 
\[X=\bigcup\big\{X_{\sigma,\rho}: (\sigma,\rho)\in J_n\cup K_n\ \wedge \
n<\omega\big\}\] 
and
\[R_n=\big\{(x,y)\in P^2:  x\neq y \mbox{ and for all }
(\sigma,\rho)\in J_n\cup K_n\mbox{ we have }\ x\notin[\sigma]\ \lor\
y\notin  [\rho]\big\}.\] 
Easily, $X$ is a meager subset opf $P$, each $R_n$ is meager in $P^2$ and
identification of variables in $R_n$ leads to an empty set (so meager). 

By \cite[Theorem 1, p. 141]{My64} there is a perfect set $P'\subseteq
P$ such that  
\begin{itemize}
\item $(P'\times P'\times P'\times P')\cap R^n_{\sigma,\rho} =
  \emptyset$ for all $n<\omega$ and $(\sigma,\rho)\in J_n\cup K_n$,
  and 
\item $P'\cap X=\emptyset$ and $(P'\times P')\cap R_n=\emptyset$ for all
  $n<\omega$.  
\end{itemize}
Clearly, if $x\neq y$ are from $P'$ and $n<\omega$, then for some 
  $N<\omega$ we have 
\[\big(\forall x',y'\in P'\big)\big((x\rest N=x'\rest N\ \wedge\ y\rest 
N=y'\rest N)\ \Rightarrow ((x,y)\in A_n\ \Leftrightarrow\  (x',y')\in 
A_n)\big).\] 
By shrinking the perfect $P'$ one can construct a perfect set
$P^*\subseteq P'$ and an increasing sequence $0=n_0<n_1<n_2<\ldots$
such that the demands in $(\heartsuit)$ are satisfied. 
\end{proof}

\section{The main result}

\begin{theorem}
\label{mainthm}
  There exists a ccc forcing notion $\bbP$ such that in $\bV^\bbP$, there is
  a $\Sigma^0_2$ set $B\subseteq \can$ with the properties that 
  \begin{enumerate}
\item[$(\spadesuit)$]  
\begin{enumerate}
  \item for some sequence $\langle h_\alpha:\alpha<\omega_1\rangle$ of
    pairwise distinct elements of $\can$ we have $|(h_\alpha+B)\cap
    (h_\beta+B)|\geq 4$, but 
\item in each perfect set $P\subseteq\can$ there are $f,g\in P$ with
  $|(f+B)\cap (g+B)|\leq 2$.
  \end{enumerate}
  \end{enumerate}
\end{theorem}

\begin{proof}
{\bf A condition\/} $p\in\bbP$ is a tuple 
\[p=\langle u,n,\bar{\eta},m_*, \bar{t},\mu,K\rangle = \langle
u^p,n^p,\bar{\eta}^p,m_*^p, \bar{t}^p,\mu^p,K^p\rangle\]
satisfying the following demands.
\begin{enumerate}
\item $\emptyset\neq u\in [\omega_1]^{<\omega}$, $0<m_*,n<\omega$, and
  $\bar{\eta}=\langle \eta_\alpha:\alpha\in u\rangle\subseteq {}^n 2$. 
\item $\bar{t}=\langle t_m:m<m_*\rangle$, each $t_m\subseteq {}^{n\geq}2$ is
  a tree with all maximal nodes of length $n$.
\item $\mu:[u]^2\longrightarrow {}^n2 \times m_*$, and if $\alpha\neq \beta$
  are from $u$ then we will write $\mu(\alpha,\beta)= \mu(\beta,\alpha)=
  (\rho_{\alpha,\beta}, \ell_{\alpha,\beta})$.  
\item If $\alpha\neq\beta$ are from $u$ then both 
  $\eta_\alpha+\rho_{\alpha,\beta}\in t_{\ell_{\alpha,\beta}}$ and
  $\eta_\beta+\rho_{\alpha,\beta}\in t_{\ell_{\alpha,\beta}}$. 
\item $K:u\longrightarrow m_*:\alpha\mapsto K_\alpha$ and $\eta_\alpha\in 
  t_{K_\alpha}$.
\item If $\alpha<\beta<\gamma$ are from $u$, then
  $\{K_\alpha,K_\gamma,\ell_{\alpha,\gamma}\} \neq \{K_\beta, K_\gamma,
  \ell_{\beta,\gamma}\}$.   
\item If $m<m'<m_*$ then $t_m\cap t_{m'}\cap {}^n 2=\emptyset$.
\item If $m<m_*$ then $t_m\cap {}^n 2\subseteq
  \{\eta_\alpha+\rho_{\alpha,\beta}:\alpha\neq \beta\ \wedge \alpha,\beta\in
  u\}\cup \{\eta_\alpha:\alpha\in u\}$.  
\item $\langle\eta_\alpha:\alpha\in u\rangle\conc\langle 
  \rho_{\alpha,\beta}: \alpha<\beta\ \wedge \alpha,\beta\in u\rangle$ is a 
  list of linearly independent vectors (in $({}^n 2,+,\cdot)$
  over $(2,+_2,\cdot_2)$); in particular they are pairwise distinct.   
\end{enumerate}

{\bf The order\/} $\leq_\bbP=\leq$ of $\bbP$ is defined by:\\
$p\leq q$ if and only if the following conditions are satisfied.
\begin{enumerate}
\item[(i)] $u^p\subseteq u^q$, $n^p\leq n^q$ and $m_*^p\leq m_*^q$. 
\item[(ii)] If $\alpha\in u^p$ then $\eta^q_\alpha\rest
  n^p=\eta^p_\alpha$. 
\item[(iii)] If $m<m_*^p$ then $t^q_m\cap {}^{n^p} 2= t^p_m\cap {}^{n^p}
  2$. 
\item[(iv)] If $\alpha\in u^p$ then $K_\alpha^p=K_\alpha^q$ and if
  $\alpha\neq \beta$ are from $u^p$, then
  $\ell_{\alpha,\beta}^p=\ell_{\alpha,\beta}^q$ and  $\rho_{\alpha,\beta}^p
  \trianglelefteq \rho_{\alpha,\beta}^q$.
\end{enumerate}

\begin{claim}
\label{cl0}
$(\bbP,\leq)$ is a partial order of size $\omega_1$.
\end{claim}

\begin{claim}
\label{cl6}
If $p\in \bbP$ and $b_0,c_0,b_1,c_1\in \bigcup\limits_{m<m^p_*} (t^p_m\cap
{}^{n^p} 2)$ are pairwise distinct and satisfy $b_0+c_0=b_1+c_1$, then for
some $\alpha\neq \beta$ from $u^p$ we have 
\[\{b_0,c_0,b_1,c_1\}=\{\eta^p_\alpha,\eta^p_\beta,
\eta^p_\alpha+\rho^p_{\alpha,\beta} , \eta^p_\beta+\rho^p_{\alpha,\beta}
\}.\]  
Also, for some $i<2$, $\{b_i,c_i\}=\{\eta^p_\alpha,\eta^p_\beta\}$ or 
$\{b_i,c_i\}=\{\eta^p_\alpha+\rho^p_{\alpha,\beta},\eta^p_\beta\}$ or 
$\{b_i,c_i\}=\{\eta^p_\alpha+\rho^p_{\alpha,\beta},\eta^p_\alpha\}$. 
\end{claim}

\begin{proof}[Proof of the Claim]
It follows from the definition of $\bbP$ (clause (8)) that $ b_0,c_0,b_1,c_1\in
\{\eta_\alpha, \eta_\alpha+\rho_{\alpha,\beta}: \alpha\neq\beta \mbox{ are from
} u^p\}$. Since, by clause (9), $\langle\eta^p_\alpha:\alpha\in
u\rangle\conc\langle \rho_{\alpha,\beta}: \alpha<\beta\ \wedge
\alpha,\beta\in u\rangle$ are linearly independent we easily get our
conclusion. 
\end{proof}

\begin{claim}
\label{cl1}
For every $N,M<\omega$ and $\delta<\omega_1$ the set 
\[Z^{N,M}_\delta=\{p\in\bbP: n^p\geq N \ \wedge \ m_*^p\geq M\ \wedge \
\delta\in u^p\}\] 
is open dense in $\bbP$.   
\end{claim}

\begin{proof}[Proof of the Claim]
Suppose that $p\in \bbP$ and let $\alpha\in \omega_1\setminus u^p$. 

Let $\langle \alpha_0,\ldots,\alpha_k\rangle$ be the increasing enumeration
of $u^p$. Set $u=u^p\cup\{\alpha\}$, $n=n^p+k+2$, and $m_*=m_*^p+k+2$. For 
$i\leq k$ let  
\[\eta_{\alpha_i}=\eta_{\alpha_i}^p\conc \langle \underbrace{
  0,\ldots,0}_{k+2} \rangle\quad\mbox{ and }\quad \rho_{\alpha_i,\alpha} =
\rho_{\alpha,\alpha_i} = \langle \underbrace{ 0,\ldots,0}_{n^p+i+1}
\rangle\conc \langle 1\rangle \conc \langle
\underbrace{0,\ldots,0}_{k-i}\rangle.\]  
We also let $\eta_\alpha=
\langle \underbrace{0,\ldots,0}\limits_{n^p} \rangle\conc \langle 1\rangle 
\conc \langle \underbrace{0,\ldots,0}\limits_{k+1} \rangle$ and we put
$\ell_{\alpha_i,\alpha}=\ell_{\alpha,\alpha_i}=m^p_*+i$ and
$K_\alpha=m^p_*+k+1$. Next, for $i\leq k$ we define
$K_{\alpha_i}=K_{\alpha_i}^p$ and for $i<j\leq k$ we let
$\rho_{\alpha_j,\alpha_i}=\rho_{\alpha_i,\alpha_j}=\rho_{\alpha_i,\alpha_j}^p\conc
\langle \underbrace{0,\ldots,0}_{k+2} \rangle $,  and
$\ell_{\alpha_j,\alpha_i}=\ell_{\alpha_i,\alpha_j}=
\ell_{\alpha_i,\alpha_j}^p$.  (So a function $\mu:[u]^2\longrightarrow {}^n2  
\times m_*$ is defined now too.) For $m<m_*^p$ let  
\[t_m=t_m^p\cup\{\sigma \conc \langle \underbrace{
  0,\ldots,0}\limits_{j} \rangle:\sigma\in t^p_m\cap 
  {}^{n^p} 2\ \wedge\ j<k+3\}\]
and for $m=m_*^p+i<m_*-1$ let $t_m=\{(\eta_{\alpha_i}+
\rho_{\alpha_i,\alpha}) \rest j, (\eta_\alpha+\rho_{\alpha_i,\alpha}) \rest
:j\leq n\}$ and $t_{m_*-1}=\{\eta_\alpha \rest j: j\leq n\}$. Finally, let
$\bar{t}=\langle t_m:m<m_*\rangle$.  

It is straightforward to verify that $q=\langle u,n,\bar{\eta},m_*,\bar{t},
\mu, K\rangle$ satisfies the demands of the definition of a condition in
$\bbP$. Moreover, $q$ is a condition stronger than $p$, and $\alpha\in
u^q$, $m^q_*>m_*^p+2$ and $n^q>n^p+2$.   

Now the Claim readily follows. 
\end{proof}

\begin{claim}
\label{cl2}
The forcing notion $\bbP$ has the Knaster property.  
\end{claim}

\begin{proof}[Proof of the Claim]
Suppose that $\langle p_\xi:\xi<\omega_1\rangle$ is a sequence of
conditions from $\bbP$. Applying the standard $\Delta$--lemma based
cleaning procedure we may find an uncountable set $A\subseteq
\omega_1$ such that $\{u^{p_\xi}:\xi\in A\}$ forms a $\Delta$--system of
finite sets and for $\xi<\zeta$ from $A$ we have:
\begin{enumerate}
\item[$(*)_1$]  $n^{p_\xi}=n^{p_\zeta}$, $m_*^{p_\xi}=m_*^{p_\zeta}$,
  $\bar{t}^{p_\xi} =\bar{t}^{p_\zeta}$,
\item[$(*)_2$] $|u^{p_\xi}|=|u^{p_\zeta}|$, $u^{p_\xi}\cap
  u^{p_\zeta}$ is an initial segment of both $u^{p_\xi}$ and
  $u^{p_\zeta}$ and $\max(u^{p_\xi} \setminus u^{p_\zeta})<
  \min(u^{p_\zeta} \setminus u^{p_\xi})$,
\item[$(*)_3$] if $\pi:u^{p_\xi}\longrightarrow u^{p_\zeta}$ is the
  order preserving bijection then for every $\alpha\in u^{p_\xi}$ we
  have 
\[K^{p_\xi}_\alpha=K^{p_\zeta}_{\pi(\alpha)}\quad \mbox{ and }\quad
\eta^{p_\xi}_\alpha=\eta^{p_\zeta}_{\pi(\alpha)}\]
and $\mu^{p_\xi}(\alpha,\beta)=\mu^{p_\zeta}(\pi(\alpha),\pi(\beta))$
for all $\alpha<\beta$ from $u^{p_\xi}$.
\end{enumerate}
We may assume that $u^{p_\xi}\cap u^{p_\zeta} \neq \emptyset \neq u^{p_\xi}  
\setminus u^{p_\zeta}$ for distinct $\xi,\zeta\in A$.

We will argue that if $\xi<\zeta$ are from $A$, then the conditions
$p_\xi,p_\zeta$ are compatible. 

Let 
\begin{enumerate}
\item[$(*)_4$] $\langle \gamma_0,\ldots,\gamma_{k_0}\rangle$ be the
  increasing enumeration of $u^{p_\xi}\cap u^{p_\zeta}$, $\langle 
  \alpha_0,\ldots,\alpha_{k_1}\rangle$ be the increasing  enumeration  of
  $u^{p_\xi}\setminus u^{p_\zeta}$ and $\langle
  \beta_0,\ldots,\beta_{k_1}\rangle$ be the increasing   enumeration  of
  $u^{p_\zeta}\setminus u^{p_\xi}$;  
\item[$(*)_5$] $k^*=(k_1+1)(k_0+k_1+3)+ \frac{(k_1-1)(k_1+2)}{2} +1$,   
  $n=n^{p_\xi}+k^*$,  and $m_*=m_*^{p_\xi}+(k_1+1)^2$;  
\item[$(*)_6$] for  $i< k^*$ let  $\nu_i=\langle
  \underbrace{0,\ldots,0}_i \rangle\conc \langle 1\rangle \conc \langle  
\underbrace{0,\ldots,0}_{k^*-i-1}\rangle\in {}^{k^*} 2$; 
\item[$(*)_7$]  $u=u^{p_\xi}\cup u^{p_\zeta}=\{\alpha_i:i\leq k_1\}\cup
  \{\beta_i:i\leq k_1\}\cup\{\gamma_i:i\leq k_0\}$; 
\item[$(*)_8$] for $i\leq k_1$ let
  $\eta_{\alpha_i}=\eta_{\alpha_i}^{p_\xi}\conc \langle \underbrace{
    0,\ldots,0}_{n-n^{p_\xi}}\rangle$, $\eta_{\beta_i}=
  \eta_{\beta_i}^{p_\zeta}\conc \nu_i$ and for $i\leq k_0$ let
  $\eta_{\gamma_i}=\eta_{\gamma_i}^{p_\xi}\conc \langle \underbrace{  
    0,\ldots,0}_{n-n^{p_\xi}} \rangle$; 
\item[$(*)_9$] $K_{\alpha_i}= K^{p_\xi}_{\alpha_i}$, $K_{\beta_i}=
  K^{p_\zeta}_{\beta_i}$ (for $i\leq k_1$) and $K_{\gamma_i}=
  K^{p_\xi}_{\gamma_i}$ (for $i\leq k_0$); 
\item[$(*)_{10}$] if $\delta<\vare$ are from $u^{p_\xi}$ then
  $\rho_{\delta,\vare} =\rho^{p_\xi}_{\delta,\vare}\conc \langle 
\underbrace{0,\ldots,0}\limits_{n-n^{p_\xi}} \rangle$ and
$\ell_{\delta,\vare}= \ell^{p_\xi}_{\delta,\vare}$; 
\item[$(*)_{11}$] if $i\leq k_0$ and $j\leq k_1$ then 
$\rho_{\gamma_i,\beta_j}= \rho^{p_\zeta}_{\gamma_i,\beta_j} \conc \nu_k$,
where $k=(k_1+1)+i(k_1+1)+j$, and $\ell_{\gamma_i,\beta_j}=
\ell^{p_\zeta}_{\gamma_i,\beta_j}$; 
\item[$(*)_{12}$] if $i,j\leq k_1$ then 
$\rho_{\alpha_i,\beta_j}= \langle
\underbrace{0,\ldots,0}\limits_{n^{p_\zeta}} \rangle\conc \nu_k$, where $k=
(k_0+2)(k_1+1)+i(k_1+1)+j$, and $\ell_{\alpha_i,\beta_j}=
m^{p_\zeta}_*+i(k_1+1)+j$; 
\item[$(*)_{13}$] if $i<j\leq k_1$ then $\rho_{\beta_i,\beta_j}=
  \rho_{\beta_i,\beta_j}^{p_\zeta} \conc \nu_k$, where 
$k=(k_1+1)(k_0+k_1+3)+\frac{i(2k_1-i+1)}{2}+(j-i-1)$, and
$\ell_{\beta_i,\beta_j}= \ell_{\beta_i,\beta_j}^{p_\zeta}$; 
\item[$(*)_{14}$] if $m<m^{p_\zeta}_*$ then 
\[\begin{array}{ll}
t_m=&t_m^{p_\zeta} \cup \{\sigma
  \conc \langle \underbrace{0,\ldots,0}\limits_k \rangle:\sigma\in
  t^p_m\cap {}^{n^{p_\zeta}} 2\ \wedge\ k\leq k^*\}\ \cup \\
&\{\eta_{\beta_i}\rest k: i\leq k_1\ \wedge\ K_{\beta_i}=m\ \wedge\
k\leq n\}\ \cup\\ 
&\{(\eta_{\gamma_i}+\rho_{\gamma_i,\beta_j})\rest k: i\leq k_0\ \wedge\
j\leq k_1\ \wedge\ \ell_{\gamma_i,\beta_j}=m\ \wedge\ k\leq n\}\ \cup \\
&\{(\eta_{\beta_j}+\rho_{\gamma_i,\beta_j})\rest k: i\leq k_0\ \wedge\ j\leq
k_1\ \wedge\ \ell_{\gamma_i,\beta_j}=m\ \wedge\ k\leq n\}\ \cup \\
&\{(\eta_{\beta_i}+\rho_{\beta_i,\beta_j})\rest k: i<j\leq k_1\
\wedge\ \ell_{\beta_i,\beta_j}=m\ \wedge\ k\leq n\}\ \cup \\
&\{(\eta_{\beta_j}+\rho_{\beta_i,\beta_j})\rest k: i<j\leq k_1\
\wedge\ \ell_{\beta_i,\beta_j}=m\ \wedge\ k\leq n\}
\end{array}\]
and
\item[$(*)_{15}$]  for  $m=m_*^{p_\zeta}+i(k_1+1)+j<m_*$, $i,j\leq k_1$, we
  let  
\[t_m=\{(\eta_{\alpha_i}{+}\rho_{\alpha_i,\beta_j})\rest k,
  (\eta_{\beta_j}{+}\rho_{\alpha_i,\beta_j})\rest k:k\leq n\}.\]
\end{enumerate}
Clauses $(*)_{10}$--$(*)_{13}$ define also  a function
$\mu:[u]^2\longrightarrow {}^n2 \times m_*$.  Finally, let  $\bar{t}=\langle
t_m:m<m_*\rangle$.    

One easily verifies that $q=\langle u,n,\bar{\eta},m_*,\bar{t},
\mu, K\rangle$ satisfies the demands of the definition of a
condition in $\bbP$ and that this condition is a common upper bound of
$p_\zeta$ and $p_\xi$.
\end{proof}

We define $\bbP$--names $\name{h}_\alpha$ (for $\alpha<\omega_1$),
$\name{T}_m$ (for $m<\omega$) and $\name{r}_{\alpha,\beta}$ (for
$\alpha<\beta<\omega_1$) by
\begin{itemize}
\item $\forces_{\bbP}$`` $\name{h}_\alpha=\bigcup\{\eta^p_\alpha: p\in 
  \name{G}_\bbP\ \wedge \ \alpha\in u^p\}$ '',
\item $\forces_{\bbP} $`` $\name{T}_m=\bigcup\{t^p_m: p\in 
  \name{G}_\bbP\ \wedge \ m<m^p_*\}$ '',
\item $\forces_{\bbP} $`` $\name{r}_{\alpha,\beta}=
  \bigcup\{\rho^p_{\alpha,\beta}: p\in \name{G}_\bbP\ \wedge \
  \alpha,\beta\in u^p\}$ ''. 
\end{itemize}

\begin{claim}
\label{cl3}
For $\alpha<\beta<\omega_1$ and $m<\omega$ we have 
\begin{enumerate}
\item $\forces_\bbP$``
  $\name{h}_\alpha,\name{r}_{\alpha,\beta}\in\can$ '',
\item $\forces_{\bbP}$`` $\name{T}_m\subseteq {}^{\omega>}2$ is a tree with
  no maximal nodes ''. 
\item $\forces_{\bbP}$`` if $m<m'<\omega$ then $[\name{T}_m]\cap
  [T_{m'}]=\emptyset$ ''. 
\end{enumerate}
\end{claim}

\begin{proof}[Proof of the Claim]
  By Claim \ref{cl1} and the definition of the order of $\bbP$. 
\end{proof}

Let $\name{B}$ be the $\bbP$--name for the $\Sigma^0_2$ subset of $\can$
given by $\forces_\bbP$`` $\name{B}=\bigcup\limits_{m<\omega} [\name{T}_m]$
''. 

\begin{claim}
\label{cl4}
For each $\alpha<\beta<\omega_1$ we have
\[\forces_\bbP\mbox{`` } |(\name{h}_\alpha +B)\cap (\name{h}_\beta+B)|\geq
4\mbox{ ''.}\] 
\end{claim}

\begin{proof}[Proof of the Claim]
  It should be clear that $\name{h}_\alpha,\name{h}_\beta,
  \name{h}_\alpha+\name{r}_{\alpha,\beta}$ and
  $\name{h}_\beta+\name{r}_{\alpha,\beta}$ are forced to belong to
  $\name{B}$ and they all are pairwise distinct. Therefore, ${\bold 0},
  \name{r}_{\alpha,\beta}, \name{h}_\alpha
  +\name{h}_\beta$ and $\name{h}_\alpha
  +\name{h}_\beta+\name{r}_{\alpha,\beta}$ are distinct elements of the
  intersection $(\name{h}_\alpha +B)\cap (\name{h}_\beta+B)$. 
\end{proof}

\begin{claim}
\label{cl5}
\[\forces_\bbP\mbox{`` for every perfect set $P\subseteq\can$ there
  are $f,g\in P$ with }|(f+\name{B})\cap (g+\name{B})|<4 \mbox{ ''.}\]    
\end{claim}

\begin{proof}[Proof of the Claim]
Suppose $G\subseteq \bbP$ is generic over $\bV$ and let us work in
$\bV[G]$ for a while. Assume towards contradiction that $P\subseteq
\can$ is a perfect set such that 
\[(\forall f,g\in P)(|(f+\name{B}^G)\cap (g+\name{B}^G)|\geq 4.\] 
Then distinct for $f,g\in P$ there are $b_0,c_0,b_1,c_1\in B$ such that  
$\{b_0,c_0\}\cap \{b_1,c_1\}=\emptyset$ and $f+g=b_0+c_0=b_1+c_1$. 

Now, for $(\ell_0,m_0,\ell_1,m_1,N)\in\omega^5$ let 
\[\begin{array}{r}
A_{\ell_0,m_0,\ell_1,m_1}^N= \{(f,g)\in P^2:
\mbox{for some } b_i\in [\name{T}_{\ell_i}^G], c_i\in
[\name{T}_{m_i}^G] \mbox{ (for $i<2$) we have }\ \\
b_0\rest N\neq c_0\rest N\mbox{ and }\{b_0\rest N,c_0\rest N\}\cap
    \{b_1\rest N, c_1\rest N\}=\emptyset\ \\
\mbox{ and }f+g=b_0+c_0=b_1+c_1\}.  
\end{array}\] 
The sets $A^N_{\ell_0,m_0,\ell_1,m_1}$ are $\Sigma^1_1$, so we may use
Lemma \ref{contread} to choose a perfect set $P^*\subseteq P$ and an
increasing sequence $0=n_0<n_1<n_2<n_3<\ldots$ such that 
\begin{enumerate}
\item[$(\boxtimes)_1$] for each $k<\omega$ and any distinct
  $x,y\in P^*$ we have:  
  \begin{enumerate}
  \item if $x\rest n_{k+1}\neq y\rest n_{k+1}$,
    $\ell_0,m_0,\ell_1,m_1,N\leq k$ then for all $x',y'\in P^*$
    satisfying $x\rest  n_{k+1}=x'\rest n_{k+1}$ and $y\rest
    n_{k+1}=y'\rest  n_{k+1}$ we have 
\[(x,y)\in A_{\ell_0,m_0,\ell_1,m_1}^N\ \Leftrightarrow\ (x',y')\in
A_{\ell_0,m_0,\ell_1,m_1}^N,\]     
\item the set $\{z\rest n_{k+1}: z\in P^*\ \wedge \ z\rest n_k=x\rest
  n_k\}$ has exactly two elements.  
  \end{enumerate}
\end{enumerate}
By our assumption on $P$ we know that 
\begin{enumerate}
\item[$(\boxtimes)_2$] for each distinct $x,y\in P^*$ there are
  $\ell_0,m_0,\ell_1,m_1,N<\omega$ such that $(x,y)\in
  A^N_{\ell_0,m_0,\ell_1,m_1}$.  
\end{enumerate}
Therefore, by induction on $j\leq 21$ we may choose $0=k_0<k_1<k_2<\ldots<
k_j< \ldots<k_{20}<k_{21}$ and $A\subseteq P^*$ such that 
\begin{enumerate}
\item[$(\boxtimes)_3$] $|A|=2^{20}$, say $A=\{x_0,\ldots,x_{2^{20}-1}\}$,
\item[$(\boxtimes)_4$] if $j\leq 20$, $x,y\in A$ and $x\rest n_{k_j}\neq y\rest
  n_{k_j}$, then $(x,y)\in A_{\ell_0,m_0,\ell_1,m_1}^N$ for some
  $\ell_0,m_0,\ell_1,m_1,N <k_j$, 
\item[$(\boxtimes)_5$] if $j<20$ and $x\in A$, then there is $y\in A$ such that
  $x\rest n_{k_j} = y\rest n_{k_j}$ but $x\rest n_{k_{j+1}} \neq y\rest
  n_{k_{j+1}}$. 
\end{enumerate}
Let $\name{P}^*$, $\name{\bar{n}}$,
$\name{A}=\{\name{x}_0,\ldots,\name{x}_{2^{20}-1}\}$, and $\name{\bar{k}}$ be
$\bbP$--names for the objects appearing in $(\boxtimes)_1$--$(\boxtimes)_5$
  and let a condition $p\in G$ force they have the properties listed there. 

Passing to a stronger condition we may also demand that 
\begin{enumerate}
\item[$(\boxtimes)_6$] $p$ decides the values of 
  $\name{k}_0,\name{k}_1,\ldots, \name{k}_{21}$, say $p\forces 
  \name{k}_j=k_j$ for $j\leq 21$,
\item[$(\boxtimes)_7$] $p$ decides the values of 
  $\name{n}_0,\name{n}_1,\ldots, \name{n}_{k_{21}}$, say $p\forces 
  \name{n}_i=n_i$ for $i\leq k_{21}$,
\item[$(\boxtimes)_8$] $p$ decides the values of $\name{x}_0\rest
  n_{k_{21}}, \ldots, \name{x}_{2^{20}-1}\rest n_{k_{21}}$, say $p\forces
  \name{x}_i\rest n_{k_{21}}=\sigma^*_i$ for $i<2^{20}$,
\item[$(\boxtimes)_9$] $n^p>n_{k_{21}}$ and $m^p_*>k_{21}$. 
\end{enumerate}
Note that it follows from $(\boxtimes)_3+(\boxtimes)_5$ that 
\begin{enumerate}
\item[$(\boxtimes)_{10}$] if $i<j<2^{20}$, then $\sigma_i^*\neq
  \sigma_j^*$. 
\end{enumerate}
Since $p$ forces that $\name{x}_i$'s have the properties listed in
$(\boxtimes)_1$ and $(\boxtimes)_3$--$(\boxtimes)_5$, there are $\sigma_i\in
{}^{n^p}2$ (for $i<2^{20}$) such that 
\begin{enumerate}
\item[$(\boxtimes)_{11}$]  $\sigma^*_i\vtl \sigma_i$ for each $i<2^{20}$,
  and  
\item[$(\boxtimes)_{12}$]  if $i,j<2^{20}$ are distinct, then for some 
  $\ell_0(i,j), m_0(i,j), \ell_1(i,j), m_1(i,j)<m^p$ and $b_0(i,j)\in
  t^p_{\ell_0(i,j)}\cap {}^{n^p} 2$, $c_0(i,j)\in t^p_{m_0(i,j)}\cap
  {}^{n^p} 2$, $b_1(i,j)\in t^p_{\ell_1(i,j)} \cap {}^{n^p} 2$, $c_1(i,j)\in
  t^p_{m_1(i,j)} \cap {}^{n^p} 2$ we have
  \begin{enumerate}
\item $\sigma_i+\sigma_j=b_0(i,j)+c_0(i,j)=b_1(i,j)+c_1(i,j)$, and 
\item $\{b_0(i,j),c_0(i,j)\}\cap \{b_1(i,j),c_1(i,j)\}=\emptyset$,
  \end{enumerate}
\item[$(\boxtimes)_{13}$]  if $i,i',j,j'<2^{20}$ and $k<n^p$ are such that
  $\sigma_i\rest k=\sigma_{i'}\rest k\neq\sigma_j\rest k=\sigma_{j'}\rest
  k$, then 
\[  \{\ell_0(i,j), m_0(i,j), \ell_1(i,j), m_1(i,j)\} = \{\ell_0(i',j'),
m_0(i',j'), \ell_1(i',j'), m_1(i',j')\}.\] 
\end{enumerate}
It follows from $(\boxtimes)_{10}$--$(\boxtimes)_{12}$ that there are no
repetitions in the list $b_0(i,j)$, $c_0(i,j)$, $b_1(i,j)$, $c_1(i,j)$.  

By Claim \ref{cl6}, for distinct $i,j<2^{20}$ we can find $\ell<2$ and
distinct $\alpha,\beta$ from $u^p$ such that 
\begin{itemize}
\item either $\{b_\ell(i,j),c_\ell(i,j)\}=\{\eta^p_\alpha,\eta^p_\beta\}$ (in which case we
  set $h(i,j)=1$), 
\item or $\{b_\ell(i,j),c_\ell(i,j)\}=\{\eta^p_\alpha+\rho^p_{\alpha,\beta},
  \eta^p_\beta\}$ (and then we set $h(i,j)=0$), 
\item or $\{b_\ell(i,j),c_\ell(i,j)\}=\{\eta^p_\alpha+\rho^p_{\alpha,\beta},
  \eta^p_\alpha\}$ (and then we also set $h(i,j)=0$).  
\end{itemize}
Note that 
\begin{enumerate}
\item[$(\boxtimes)_{14}$] if $i,j,k<2^{20}$ are pairwise distinct, then
  $h(i,j)=1$ or $h(j,k)=1$ or $h(i,k)=1$. 
\end{enumerate}
Why? First suppose that for some $\alpha<\beta$, $\gamma<\delta$ and
$\vare<\zeta$ from $u^p$ we have 
\[\begin{array}{l}
\sigma_i+\sigma_j=b_0(i,j)+c_0(i,j)=\eta^p_\alpha+\eta_\beta^p+
\rho_{\alpha,\beta}^p ,\\
\sigma_j+\sigma_k=
b_0(j,k)+c_0(j,k)=\eta^p_\gamma+\eta_\delta^p+\rho_{\gamma,\delta}^p,\\
\sigma_i+\sigma_k= b_0(i,k)+c_0(i,k)=\eta^p_\vare+\eta_\zeta^p+
    \rho_{\vare,\zeta}^p.
\end{array}\]
Then 
\[0=(\eta^p_\alpha+\eta_\beta^p+
\rho_{\alpha,\beta}^p) +
(\eta^p_\gamma+\eta_\delta^p+\rho_{\gamma,\delta}^p) +
(\eta^p_\vare+\eta_\zeta^p+ \rho_{\vare,\zeta}^p).\]  
However, by the linear independence, it is not possible (the $\rho$'s
cannot be cancelled).

Second, suppose 
\[\begin{array}{l}
\sigma_i+\sigma_j=b_0(i,j)+c_0(i,j)=\eta^p_\alpha+\eta_\alpha^p+
\rho_{\alpha,\beta}^p =\rho^p_{\alpha,\beta},\\
\sigma_j+\sigma_k=
b_0(j,k)+c_0(j,k)=\eta^p_\gamma+\eta_\delta^p+\rho_{\gamma,\delta}^p,\\
\sigma_i+\sigma_k= b_0(i,k)+c_0(i,k)=\eta^p_\vare+\eta_\zeta^p+
    \rho_{\vare,\zeta}^p.
\end{array}\]
Then $\rho^p_{\alpha,\beta}=(\eta^p_\gamma+\eta_\delta^p+
\rho_{\gamma,\delta}^p) + (\eta^p_\vare+\eta_\zeta^p+
\rho_{\vare,\zeta}^p)$, and this is again not possible by the linear independence.

Thirdly, the assumption that $\sigma_i+\sigma_j=\rho^p_{\alpha,\beta}$,
$\sigma_j+\sigma_k= \rho^p_{\gamma,\delta}$ and $\sigma_i+\sigma_k=
\eta^p_\vare+\eta_\zeta^p+ \rho_{\vare,\zeta}^p$ leads to 
\[\rho^p_{\alpha,\beta}+\rho^p_{\gamma,\delta}=\eta^p_\vare+\eta_\zeta^p+
\rho_{\vare,\zeta}^p,\]
again clear contradiction.

Finally, the configuration $\sigma_i+\sigma_j=\rho^p_{\alpha,\beta}$,
$\sigma_j+\sigma_k= \rho^p_{\gamma,\delta}$ and $\sigma_i+\sigma_k=
\rho_{\vare,\zeta}^p$ is also impossible. 
\medskip

Using Lemma \ref{get4} we may find $A\subseteq \{i:i<2^{20}\}$ such
that 
\begin{enumerate}
\item[$(\boxtimes)_{15}$] 
\begin{enumerate}
\item$|A|\geq 5$,  and $\cA=\{\sigma_i:i\in A\}$ contains a 4--arrangement
  (see \ref{nice4}), and    
\item $A$ is $h$-homogeneous in color $1$, i.e., $h(i,j)=1$ for 
  $i<j$ from $A$
\end{enumerate}
\end{enumerate}
(remember $(\boxtimes)_3+(\boxtimes)_5$). Now, $(\boxtimes)_{15}$(b) implies
that 
\[\cA+\cA\subseteq \{\eta^p_\alpha:\alpha\in u^p\}+\{\eta^p_\alpha:
\alpha\in u^p\}.\] 
Hence, by Lemma \ref{litlem}(b), there is
$x\in {}^{n^p} 2$ such that $\cA+x\subseteq \{\eta^p_\alpha:\alpha\in 
u^p\}$. Since $\cA+x$ contains a $4$--arangement we may find
$\alpha<\beta<\gamma$ such that  $\eta^p_\alpha,\eta^p_\beta,\eta^p_\gamma
\in \cA+x$ and 
\[\begin{array}{r}
\min\{k<n^p:\eta^p_\alpha(k)\neq \eta^p_\gamma(k)\}=
\min\{k<n^p:\eta^p_\beta(k)\neq \eta^p_\gamma(k)\}<\\
\min\{k<n^p:\eta^p_\alpha(k)\neq  \eta^p_\beta(k)\}.
\end{array}\]
Now, $(\eta^p_\alpha+x), (\eta^p_\beta+x), (\eta^p_\gamma+x) \in \cA$ so let 
$i,j,k<2^{20}$ be such that $\eta^p_\alpha+x=\sigma_i,
\eta^p_\beta+x=\sigma_j$ and $\eta^p_\gamma+x=\sigma_k$. By
$(\boxtimes)_{13}$ we have  
\[  \{\ell_0(i,k), m_0(i,k), \ell_1(i,k), m_1(i,k)\} = \{\ell_0(j,k),
m_0(j,k), \ell_1(j,k), m_1(j,k)\}\] 
but this implies that 
\[\big\{K^p_\alpha,K^p_\gamma,\ell^p_{\alpha,\gamma}\big\} =
\big\{K^p_\beta,K^p_\gamma,\ell^p_{\beta,\gamma}\big\},\]
contradicting clause (6) of the definition of $\bbP$. 
\end{proof}
\end{proof}

For the completeness of the picture let us note that, consistently, there is no
Borel set satisfying \ref{mainthm}$(\spadesuit)$.

\begin{theorem}
\label{2coh}
Assume CH. Let $\bbC_{\omega_2}$ is the forcing notion adding $\omega_2$  
Cohen reals. Then in $\bV^{\bbC_{\omega_2}}$ the following holds: 
\begin{quotation}
If $B\subseteq \can$ is Borel , $\langle \eta_\alpha:\alpha<\omega_2 \rangle 
\subseteq \can$ and 
\[\big(\forall \alpha<\beta<\omega_2\big)\big(|(B+\eta_\alpha)\cap
(B+\eta_\beta)|\geq 4\big), \]
then there is a perfect set $P\subseteq \can$ such that 
\[\big(\forall x,y\in P\big)\big(|(B+x)\cap (B+y)|\geq 4 \big).\]
\end{quotation}
\end{theorem}

\begin{proof}
Straightforward argument; it also follows from Shelah \cite[Fact
1.16]{Sh:522}. See more in Ros{\l}anowski and Shelah \cite{Sh:F1444}. 
\end{proof}


\end{document}